\documentclass[12pt]{amsart}
\usepackage{amscd}
\usepackage{amssymb}

\newcommand{\bZ}{{\mathbb Z}}

\newcommand{\bC}{{\mathbb C}}

\newcommand{\la}{{\langle}}
\newcommand{\ra}{{\rangle}}

\newtheorem{thm}{Theorem}[section]
\newtheorem{lemma}[thm]{Lemma}
\newtheorem{cor}[thm]{Corollary}

\newtheorem{conj}[thm]{Conjecture}

\numberwithin{equation}{section}

\begin{document}

\title[]{Chow ring and BP-theory of the extraspecial $2$-group of order $32$}
 
\author{N.Yagita}

\address{ faculty of Education, 
Ibaraki University,
Mito, Ibaraki, Japan}
 
\email{ yagita@mx.ibaraki.ac.jp, }


\begin{abstract}
 We write down the mod $2$ Chow ring of the classifying space of
$G=2_+^{1+4}=D_8\cdot D_8$,  which has nilpotent elements.
\end{abstract}

\maketitle
   
\section{Introduction}

Let $p$ be a prime number.
Let $G$ be a $p$-group and $BG$ its classifying space.
Let us write simply  by $H^*(G;\bZ/p)=H^*(BG;\bZ/p)$ the mod $p$ 
cohomology of the group $G$, and 
by $CH^*(G)=CH^*(BG)$ the Chow ring of the classifying space
$BG$ over the complex number field $\bC$.

In this paper, we write down the (most ease) case where
$CH^*(G)/2$ has nonzero nilpotent elements
(but $H^*(G;\bZ/2)$ has not).
Note that  Chow rings $CH^*(G)/p$ 
for all $G$ with $|G|\le p^4$ are 
still computed by Totaro in [To2].
Let $D(2)=2_+^{1+4}=D_8\cdot D_8$ be the extraspecial $2-$group (of order $2^5$)
which is the central product of two dihedral groups $D_8$.
\begin{thm}
 There are ring isomorphisms 
\[CH^*(D(2))/2\cong (H^*(D(2);\bZ/2))^2\oplus \bZ/2[c_4]\{t''\}\]
\[ \cong (\bZ/2[y_1,y_2,y_3,y_4]/(q_0',q_1')
\oplus \bZ/2\{t''\})
\otimes \bZ/2[c_4]\]
where $deg(y_i)=1,$ $deg(c_4)=4$, $deg(t'')=2$,  and $q_0'=y_1y_2+y_3y_4$, 
\[q_1'=Sq^2(q_0')=y_1^2y_2+y_1y
_2^2+y_3^2y_4+y_3y_4^2.\]
  The multiplications are given
$(t'')^2=y_it''=0$ for all $1\le i\le 4$.
\end{thm}

Let $BP^*(G)=BP^*(BG)$ be the Brown-Peterson  theory
with the coefficient $BP^*=\bZ_{(p)}[v_1,v_2,...]$ and 
$|v_i|=-2(p^i-1)$ (for detalis of the $BP$-theory, see [Ha]
or [Ra]).
We also show the mod $2$ Totaro conjecture ([To1]) ;
\begin{thm}
The mod $2$ Totaro conjecture holds for $D(2)$, that is 
 \[CH^*(D(2))/2\cong BP^*(D(2))\otimes_{BP^*}\bZ/2.\]
\end{thm}

Let us write by $\Omega^*(G)$ the $BP$-version $\Omega^*(BG)\otimes_{MU^
*}BP^*$
of the algebraic cobordism $\Omega^*(BG)=MGL^{2*,*}(BG)$
([Vo1,2], [Le-Mo1,2]).
Let $t_{\bC}: \Omega^*(X)\to BP^{2*}(X(\bC))$ be the realization map.
There is a conjecture such that ;
\begin{conj} 
The  realization map $t_{\bC}$ is an isomorphism for
each algebraic group $G$, e.g. $\Omega^*(BG)\cong BP^*(BG)$. \end{conj}

It is known the above conjecture is true for
 connected groups ([To1], [Ya2,3]);
$O_n,\ SO_n, PGL_p,G_2,Spin_7$.
As for finite groups $G$,  the above conjecture is known to be true for
abelian groups and the extraspecial $p$-groups of order $p^3$, i.e.  $p_{+}^{1+2},p_{-}^{1+2}$
for all primes [Ya4].  While the author can not see
this conjecture for $D(2)$, 
 in the last section, we add some notes 
 for groups satisfying the above conjecture.

\section{The Chow ring of $D(2)$}

The group $D(2)$ is isomorphic to the extraspecial $2$-group
$2_+^{1+2}$,
which  has the central extension
\[ 1\to N\to D(2)\to Q\to 1,\quad N\cong \bZ/2,\ Q\cong (\bZ/2)^4.\]
We use notations such that $N\cong \la c\ra, Q\cong \la a_1,a_2,a_3,a_4\ra$ and 
\[ D(2)\cong \la a_1,...,a_4,c|a_1^2=...=a_4^2=c^2=1,\]
\[ \qquad \qquad  [a_1,a_2]=[a_3,a_4]=c=(a_1a_2)^2=(a_3a_4)^2\ra.\]

The mod $2$ cohomology is given by Quillen [Qu1]
\[ H^*(D(2);\bZ/2)
\cong \bZ/2[x_1,x_2,x_3,x_4]/(q_0,q_1)
\otimes \bZ/2[w_4]\]
where $q_0=x_1x_2+x_3x_4$ and $q_1=Sq^1q_0=x_1^2x_2+x_1x_2^2+x_3^2x_4+x_3x_4^2$.
Here $x_i$ (and $w_4$) are  Stiefel-Whitney classes for some real representations, and hence the powers are Chern classes,
that is,
\[y_i=x_i^2=c_1(e_i),\quad 
e_i:D(2)\to \la a_i\ra \to \bC^\times\]
where $e_i$ is the nonzero linear representation, and  
\[ c_4=(w_4)^2=c_4(\eta) ;\quad  \eta=Ind_H^D(e), \]
where $H=\la c,a_1,a_3\ra$ is the maximal elementary abelian $2$-subgroup of $D(2)$
and $e:H \to \la c\ra \to \bC^\times$ is a nonzero linear representation.
We note that $H^*(D(2);\bZ/2)$ has no nonzero nilpotent elements ([Qu1]).

It is well known (e.g., [Qu1]) that each irreducible representation of an extraspecial $p$-group $P$
is a linear representation or 
just one induced representation of a linear representation of a maximal elementary abelian $p$-group of $P$.
Hence the Chern subring (the subring of $H^*(D(2);\bZ/2)$ multiplicatively generated by Chern classes)
is
\[ Ch(H^*(D(2);\bZ/2))\cong H^*(D(2);\bZ/2)^2\]
\[\cong \bZ/2[y_1,...,y_4]/(q_0',q_1')\otimes \bZ/2[c_4]\]
where $q_0'=y_1y_2+y_3y_4$ and $q_1'=Sq^2q_0'=y_1^2y_2+y_1y_2^2+y_3^2y_4+y_3y_4^2$.

Now we start to consider the Chow ring of $BD(2)$. 
In this paper we write $CH^*(BD(2))$ by $CH^*(D(2))$ (we also write $BP^*(BD(2))$ by $ BP^*(D(2))$).

Moreover we note following facts (see [To1] for details).  By the Rieman-Roch theorem without denominator,
$CH^2(D(2))/2$ is generated by $2nd$ Chern classes (of some representations), that means,
it is generated by  $y_iy_j$ and $c_2(\eta)$.  

\begin{lemma} We have  $q_0'=y_1y_2+y_3y_4=0 \in CH^2(D(2))/2$ and
\[ CH^2(D(2))/2\cong \bZ/2\{y_iy_j|1\le i,j\le 4\}/(q_0')\oplus \bZ/2
\{c_2(\eta)\}\]
where $A\{a,b,..\}$ means the free $A$-module 
generated by $a,b,...$.
\end{lemma}
\begin{proof}
By Totaro (Corollary 3.5 in [To1] or Lemma 15.1 in [To2]), 
the integral cycle map
\[cl_{int}: CH^2(X)_{(2)}\to H^4(X;\bZ_{(2)})\]
is injective.  The higher $2$-torsion of the integral cohomology of 
extraspecial $2$-groups are studied by Harada-Kono 
([Ha-Ko], [Sc-Ya1]).
Let $C(2)^*=H^*(D(2))/J_Q$ where $J_Q$ is the ideal generated by the 
image of $H^*(Q)$ in $H^*(D(2))$ (for $Q\cong (\bZ/2)^4$). Then
Harada-Kono show  that 
\[\bZ/2^{s(*)}\cong C(2)^*\subset H^*(D(2)),\]
and  when $*=4m$, we have  $C(2)^*\cong \bZ/8$.
Let $w_4$ be a generator of $C(2)^4$.  Then it is known
\[ w_4|N=u^2\quad (w_4|N'=(u')^2\ \ for\ N'=\la a_1a_2\ra \cong \bZ/4)\]
identifying $H^*(N)\cong \bZ[u]/(2u)$ and $H^*(N')\cong
\bZ[u']/(4u')$. 

On the other hand, all elements in $J_Q$ are just $2$-torsion.
Moreover $q_0'$ is (zero or) just $2$-torsion (since so are $y_i$).
Therefore we get
\[ cl_{int}(q_0')=4\lambda w_4\quad 
for \ some \ \lambda \in \bZ/8.\]

Let $c(\eta)=\sum c_i(\eta)$ is the total Chern class.
Then we see
\[c(\eta)|_{N'}=(1+u')^4=1+4u'+6(u')^2+...\quad mod(8).\]
Hence $c_2(\eta)|_{N'}=-2(u')^2$ and so
 $q_0'=-2\lambda c_2(\eta)$.
\end{proof}

 We recall a theorem of Totaro.
\begin{thm}
(Theorem 11.1 in [To2])
Let $P$ be a $p$-group  such that $P$ has a faithful complex representation of
dimension at most $p+2$.
Then the mod $p$ Chow ring of $BP$ consists of transferred Euler classes.
\end{thm}

First note that Euler classes of $CH^*(D(2))$ are (multiplicatively) generated by $y_1,...,y_4$ and 
$c_4(\eta)$. 
 Next we consider the transfer images. Each proper maximum subgroup $M$ of $D(2)$ is isomorphic to
$D_8\oplus \bZ/2$, and let it be $\la a_1,a_2,c,a_3\ra$.  The Chow ring $CH^*(M)/2$
is generated by Chern classes
\[ y_1,y_2,y_3,\quad and\quad c_2=c_2(\eta')\]
where $\eta'=Ind_{H}^{M}(e)$ and recall that  $e: H=\la c,a_1,a_3\ra \to \bC^{\times}$.
Let us write the transfer
$ t_2=Tr_M^{D(2)}(c_2).$
We note (by the double coset formula)
$t_2|_{N'=\la a_1a_2\ra}=2(u')^2$  identifying
 $CH^*(N')\cong \bZ[u']/(4u')$ where $N'\cong \bZ/4$.
Therefore $t_2=c_2(\eta)\ mod(y_iy_j)$ in $CH^*(D(2))/2$
from Lemma 2.1.
Of course
$ Tr_M^{D(2)}(y_ic_2)=y_it_2$ for all $1\le i\le 4.$

 For an other proper 
maximal subgroup $\tilde M$,  we similarly have
 the transfer $\tilde t_2$. However we have
\[ \tilde t_2=c_2(\eta)=t_2\quad mod(y_iy_j).\]

From the Totaro theorem (Theorem 2.2), we have ;
\begin{lemma}
The mod $2$ Chow ring $CH^*(D(2))$ is multilpicatively generated by
  $ y_1,...,y_4,$ $c_4=c_4(\eta)$ and $t_2$ (or $c_2(\eta))$.
\end{lemma}

Next we study the nilpotent elements.
Let us write by $cl$ the mod $2$ cycle map
\[cl:CH^*(D(2))/2\to H^*(D(2);\bZ/2).\]
Recall that the Chern subring of $H^*(D(2);\bZ/2)$ is generated by $y_i$ and $c_4(\eta)$.  Since $t_2$ is a Chern class, we can take
$y\in \bZ[y_1,...,y_4]$ such that $cl(t_2)=y\in  H^*(D(2);\bZ/2)$.

Let $t''=t_2-y$ in $CH^*(D(2))$ so that $cl(t'')=0$ and $t''$
is a (nonzero) nilpotent element in $CH^*(D(2))$
because  $Ker(t_{\bC})$ is nilpotent, since   
$t_{\bC}$ is $F$-isomorphic  from the Quillen theorem
for Chow rings [Ya2]. (Note $t''$ is nonzero in $CH^*(D(2))/2$
because $t''|_{N'}=2(u')^2$ and $CH^2(D(2))|_{N'}$ is generated by
$2(u')^2$.)

\begin{lemma}  $y_4t''=0$ in $CH^*(D(2))/2$.
\end{lemma}
\begin{proof}
Note that
\[y_4t''=y_4(t_2-y)=tr_{M}^{D(2)}(y_4|_M\cdot c_2)-y_4y=-y_4y,\]
where we used $y_4|_M=0$.  Note $y_4t''$ is nilpotent but
$H^*(D(2);\bZ/2)$ has no nonzero nilpotent element.
Hence $y_4y\in (q_0',q_1')$ and also zero in $CH^*(D(2))/2$.
Thus $y_4t''=0$ in $CH^*(D(2))/2$.
(Since $CH^*(X)$ has the reduced power operation $Sq^2$,
we have  $q_1'=Sq^2(q_0')=0$ also in $CH^*(D(2))/2$ [Vo3].)
\end{proof}
\begin{lemma} For all $1\le i\le 4$, we have $y_it''=0$.
\end{lemma}
\begin{proof}
In $CH^2(D(2))/2$, nilpotent elements generate just one dimensional $\bZ/
2$-space
$\bZ/2\{t''\}$.  Hence $t''$ is invariant under an action of
the outer automorphism $Out(D(2))$. This outer automorphism contains
\[f: a_3 \leftrightarrow  a_4, c\mapsto c,\qquad
g: a_1\mapsto a_3,\ a_2\mapsto a_4,\ c\mapsto c.\]
We have $0=f^*(y_4t'')=y_3t''$ and $0=g^*(y_4t'')=y_2t''$.
\end{proof}
\begin{lemma}  $(t'')^2=0$ in $CH^*(D(2))/2$.
\end{lemma}
\begin{proof}  We compute
\[ (t'')^2=t''(tr_M^{D(2)}(c_2)-y)=t''tr_M^{D(2)}(c_2)=tr_M^{D(2)}(t''|_{M}\cdot c_2)=0,\]
since $t''|_{M}$ is nilpotent but $CH^*(M)/2$ has no non zero nilpotent element.
\end{proof}
From the above lemmas, we get Theorem 1.1 in the introduction. 

{\bf Remark.}
From Theorem  in [To2], we see the topological nilpotency is
$d_0(CH^*(D(2))/2)\le 3$.
This means $y_iy_jt''=0$. So we see a bit stronger  result  $d_0(CH^*(D(2))/2)=2$
 in the above lemma.

\section{BP-theory}

By Schuster-Yagita [Sc-Ya2], it is known that the Morava $K$-theory 
$K(n)^*(BD(2))$ is generated by even dimensional elements
(see also Schuster [Sc] or Bakladze-Jibradze [Ba-Ji])
 for all  $n\ge 0$.  This implies that
$BP^*(D(2))$ is generated by even dimensional elements, and 
satisfies the condition of the Landweber exact functor theorem.

Moreover $D(2)$ is $K(n)^*$-good, namely,
$K(n)^*(D(2))$ is generated by transferred Euler classes
for all $n$.  It is known ([Ra-Wi-Ya]) that it implies that
$D(2)$ is $BP^*$-good, i.e., $BP^*(D(2)$ is generated also
by transferred Euler classes.

Recall the exact sequence
\[ (*)\quad 0\to M\cong D_8\oplus \bZ/2\to D(2)\to \bZ/2\to 0\]
%which induces the Atiyah-Hirzebruch spectral sequence
%\[ E_2^{*,*'}\cong H^*(\bZ/2;BP^*(M))\Longrightarrow
% inBP^*(D(2)).\]
Here we use notations $M=\la a_1,a_2,c,a_3\ra$ and $\bZ/2\cong\la a_4\ra $ in the following proof.

\begin{proof}[Proof of Theorem 1.2.]
The cycle map is decomposed as 
\[cl: CH^*(X)/2
\stackrel{cl_{BP}}{\to}BP^*(X)\otimes_{BP^*}\bZ/2
\stackrel{\rho}{\to} H^*(X;\bZ/2)\]
where $cl_{BP}$ is the Totaro cyle map and $\rho$ is the Thom map.

By the $BP^*$-goodness of  $D(2)$,
we see that  $cl_{BP}$ is surjective.
 Moreover it is known ([Ya2]) that $cl_{BP}$ is an $F$-isomorphism.
Hence $Ker(cl_{BP})$ is nilpotent.
Thus  it is only need to show
\[\bZ/2[c_4]\{t''\}
\subset BP^*(D(2))\otimes_{BP^*}\bZ/2.\]
(Note that $t''$ exists in $BP^*(D(2)$, but we need 
to see $t''\not =0$  and $t''$ generates a $\bZ/2[c_4]$-free module.)

Note $t''|_{N'}=2(u')^2$ and so $t''|_M$ is not a $BP^*$-module generator of $BP^*(M)$ but
$c_2(\eta')\not \in BP^*(M)^{\la a_4\ra}$.  Hence  
$t''|_M$ is a $BP^*$-module generator of $BP^*(M)^{\la a_4\ra}$.
Then we have
\[BP^*(D(2))
\otimes_{BP^*}\bZ/2
\stackrel{res}{\to} 
BP^*(M)^{\la a_4\ra}\otimes_{BP^*}\bZ/2 \supset \bZ/2[c_4]\{t_2''|_M\}.\]
The last inclusion follows from the restriction to $N'=\la a_1a_2\ra\cong \bZ/4$, 
\[ BP^*[c_4](t'')|_{N'}=BP^*[(u')^4]\{2u'\}\subset BP^*(N')\cong BP^*[u']([4](u')).\]
Thus  we have the theorem.
\end{proof}
In this paper, we do not explicitly use the following lemma
and corollary, but we note them. 
\begin{lemma}
The restriction map
\[ res: BP^*(G)\to Lim_{G\supset A:abelian}BP^*(A)\]
is an $F$-isomorphism (i.e., its kernel and cokernel
are nilpotent).
\end{lemma}
\begin{proof}
We can define the Evens norm for $BP^*$-theory.
Hence $res$ is $F$-surjective from the arguments in the proof of  Lemma 2.4 in [Qu2].  The $F$-injective follows from the arguments (3.10) in page 371 in [Qu2]. 
\end{proof}
Note that $A$ ranges
all abelian subgroups of $G$ for the $F$-injectivity.
In fact, the kernel of $BP^*(\bZ/4)\cong BP^*[u']/([4](u'))\to BP^*(\bZ/2)$ is
the ideal $[2](u')$ which is not nilpotent.
\begin{cor}  $BP^*(D(2))\subset Lim_{D(2)\supset A:abel.
}BP^*(A)$.
\end{cor}

\section{algebraic cobordism $\Omega^*(P)$}

Let $p$ be a fixed prime number.
For a smooth variety $X$ over the complex field $\bC$, let us write by
\[\Omega^*(X)= MGL^{2*,*}(X)\otimes_{MU^*}BP^*\cong ABP^{2*,*}(X)\]
the ($BP^*$-version of) algebraic cobordism defined by Voevodsky ([Vo1,2]) and
Levine-Morel ([Le-Mo1,2]).  There is a conjecture (Conjecture 1.3)
such that  the realization map induces the isomorphism
$ t_{\bC}: \Omega^*(BG)\cong   BP^*(BG)$
for the classifying space $BG$ of each  algebraic group $G$.

It is known that this conjecture is true for connected groups [Ya2,3]\\
$O_n,\ SO_n, PGL_p,G_2,Spin_7$.
As for finite groups $G$, it is known that  the conjecture is true for
abelian groups and the extraspecial $p$-groups $p_{+}^{1+2},p_{-}^{1+2}$
for all primes [Ya4].  In this section, we show the conjecture for other  $p$-groups.

 We consider a $p$-group $G$ and  its subgroup $M$ of 
index $p^s$, namely, there is the  extension
\[(*)\quad  1\to M\to G\to \bZ/p^s\to 0\]
and consider the induced spectral sequence
\[ E_2^{*,*'}\cong H^*(\bZ/p^s;BP^*(M))\Longrightarrow BP^*(G).\]
Let the right hand side group $\bZ/p^s$ in $(*)$ be generated by
$a$.  Let  $ N=1+a^*+...+(a^{p^s-1})^*$ and recall that
\[ E_2^{*,*'}\cong \begin{cases} Ker(1-a^*)\cong BP^*(M)^{\la a\ra}\quad 
 *=0\\
     Ker(1-a^*)/Im(N)\quad  *=even>0\\
Ker N/Im(1-a^*)\quad *=odd.
\end{cases} \]
We consider the cases that $E_2^{odd,*'}\cong 0$.

\begin{lemma}
Let $G$ be a $p$-group with the extension $(*)$
such that $E_2^{odd,*'}=0$.  Moreover we assume ;

(1)\quad The $mod(p)$ Totaro conjecture holds for $G$, i.e.
\[CH^*(G)/p\cong BP^*(G)\otimes _{BP^*}\bZ/p,\]

(2)\quad The conjecture 1.3 holds for $M$,
i.e.
$t_{\bC}:\Omega^*(M)\cong BP^*(M)$. \\
Then Conjecture 1.3 holds for $G$, namely,
$t_{\bC}:\Omega^*(G)\cong BP^*(G)$.
\end{lemma}
\begin{proof}
Let $y$ be the first Chern class of a nonzero linear representation for 
$G$ ;
$G\to \la a\ra \to \bC^*$.
Then from $E_2^{odd,*'}\cong 0$,  we see \[E_{\infty}^{*,*'}\cong E_{\infty}^{even,*'}\cong E_2^{even,*'}.\]
Hence we get 
\[ grBP^*(G)\cong BP^*(M)^{\la a\ra}\oplus (BP^*(M)^{\la a\ra}/N)[y]^+.\]

On the other hand, from (1), the algebraic cobordism
$\Omega^*(G)$ is also generated by $BP^*(M)^{\la a\ra}$($\cong \Omega^*(M)
^{\la a\ra}$)
and $y\in \Omega^2(G)$.
We consider the filtration defined by the $ideal(y)\subset \Omega^*(M)$.

For $x\in \Omega(M)^{\la a\ra}$, take $\tilde x\in \Omega^*(G)$ with
$\tilde x|_M=x$ (which is only decided with modulo $Ideal(y))$.
(Note we can take $\tilde{Nx}=Tr_M^{G}(x)$.)
Then $\Omega^*(G)$ is additively generated by $\tilde x$ and 
$\tilde xy^i$.
Hence we have 
\[ gr\Omega^*(G)\cong BP^*(M)^{\la a\ra}\oplus \oplus _{i\ge 1}
(BP^*(M)^{\la a\ra}/N_i)\{y^i\}\]
where $N_1\subset N_2\subset...$.
Note  $N_i\subset Im(N)$, since we have the 
cycle map
$gr\Omega^*(G)\to grBP^*(G)$.
Hence we only need to prove $N_1=Im(N)$.

For $x\in \Omega^*(M)$, we see
\[y\tilde N(x)=yTr_M^G(x)=Tr_M^G((y|_M)\cdot x)=0\quad in\ \Omega^*(G).\]
Thus $Im(N)\subset N_1$ and we see $N_i=Im(N)$ for all $i$.
\end{proof} 

For $G=D(2)$, we consider the exact sequence $(*)$
in $\S 3$,  and the induced spectral sequence
converging to $BP^*(D(2))$.
\begin{cor}
If  $E_2^{odd,*'}=0$ for the above spectral sequence,
then   Cojecture 1.3 holds for $D(2)$.
\end{cor}

Next we consider groups $P$ with $rank_p=2$ and $p\ge 3$. 
At first, we consider a split metacyclic group.  It is written
\[P=M(\ell,m,n)=\la a,b|a^{p^m}=b^{p^n}=1, [a,b]=a^{p^{\ell}}\ra\]
for $m>\ell\ge max(m-n,1)$.  Consider the extension
\[ 1\to \la a\ra \to P \to \la b\ra \to 1.\]
Then this extension satisfies the assumption in Lemma 4.1
except for (1) ([Te-Ya2]) and 
$BP^*(P)\otimes_{BP^*}\bZ_{(p)}\cong H^{even}(P;\bZ_{(p)}).$
Moreover when $m-\ell=1$,  Totaro showed
the above cohomology is isomorphic to the Chow ring
$CH^*(P)$ [To2].  Therefore we have
\begin{cor}  Conjecture 1.3 holds for $M(m,\ell,n)$ with  $m-\ell=1$.
\end{cor}

 We consider the other $rank_pP=2$ groups.
For $p\ge 5$,  groups $P$ with $rank_pP=2$ are classified by Blackburn (see [Ya1] ).
They are metacyclic groups,   and some groups $C(r)$, $G(r',e)$.
The group $C(r),\ r\ge 3$ is defined by
\[ C(r)=\la a,b,a| a^p=b^p=c^{p^{r-2}}=1,\ [a,b]=c^{p^{r-3}}\ra\]
for $r\ge 3$ so that $C(3)=p_+^{1+2}$. 
  The group $G=G(r,e), r\ge 4$ (and $e\not =0$ is a quadratic nonresidue mod $p$) is defined as
\[\la a,b,c|a^p=b^p=c^{p^{r-2}}=[b,c]=1, [a,b^{-1}]=c^{ep^{r-3}},[a,c]=b\ra.\]
The subgroup $\la a,b,c^{p}\ra$ is isomorphic to $C(r-1)$.
\begin{cor}  Conjecture 1.3 holds for $C(r),D(r+1,e)$.\end{cor}
\begin{proof}
It is known
$ CH^*(P)/p\cong H^{even}(P;\bZ)/p\cong BP^*(P)\otimes_{BP^*}\bZ/p.$ 
Here the first isomorphism is proved in [To2] and the second is shown in [Ya1].
The extension 
\[ 1\to \la c,a\ra \to (r)\to \la b\ra \to 1\]
satisfies [Ya1] the assumption Lemma 4.1 for $C(r)$
The extension
\[1\to \la a,b,c^p\ra \to G(r+1,e)\to \la c\ra \to 1\]
satisfies [Ya1] the assumption of Lemma 4.1  for $G(r+1,e)$.
\end{proof}
We write down the result for $p$-Sylow subgroups $\bZ/p\wr...\wr\bZ/p$ of symmetric groups.
Here $\bZ/p\wr X=\bZ/p \rtimes (X)^{\times p}$ is the $p$-th wreath product.
\begin{cor}  Conjecture 1.3 holds for  $\bZ/p\wr... \wr\bZ/p$.
\end{cor}
\begin{proof}
Totaro's conjecture is still proved in [To1].
We consider the extension
\[1\to (G')^p\to \bZ/p\wr G'\to \bZ/p\to 1\]
and induced spectral sequence converging to
$BP^*(\bZ/p\wr G')$.  It is proved in Lemma 5.3 in [Te-Ya2] that
if there exist $BP^*$-module generators  $\{x_i\}$
of $BP^*(G')$ such that $\{\rho(x_i)\}$ is a subset of $\bZ/p$-basis of
$H^*(G')/p$, then $E_2^{odd,*'}=0$.
By induction on the number of the wreath product, we can show the corollary.
\end{proof}

\end{document}